\documentclass[12pt]{article}
\usepackage{graphicx,url}
\usepackage{amsmath,amssymb,latexsym,amsthm}

\def\D{\mathcal{D}}

\def\Pr{\mathbf{Pr}}
\def\Q{\mathcal{Q}}

\def\tr{\textrm}
\def\Vd{V_d}
\def\Vn{V_n}

\newtheorem{theorem}{Theorem}
\newtheorem{question}[theorem]{Question}

\newtheorem{lemma}[theorem]{Lemma}
\newtheorem{prop}[theorem]{Proposition}
\begin{document}
\title{Vertex Ramsey problems in the hypercube}
\author{John Goldwasser\thanks{Department of Mathematics, West Virginia University,  Morgantown, WV 26506-6310. Email: jgoldwas@math.wvu.edu} \and John Talbot\thanks{Department of Mathematics, University College London, WC1E 6BT, UK. 
Email: j.talbot@ucl.ac.uk.  This author is a Royal Society University Research Fellow.}}

\maketitle
\begin{abstract} If we 2-color the vertices of a large hypercube what monochromatic substructures are we guaranteed to find? Call a set $S$ of vertices from $\Q_d$, the $d$-dimensional hypercube, \emph{Ramsey} if any 2-coloring of the vertices of $\Q_n$, for $n$ sufficiently large, contains a monochromatic copy of $S$.  Ramsey's theorem tells us that for any $r\geq 1$ every 2-coloring of a sufficiently large $r$-uniform hypergraph will contain a large monochromatic clique (a complete subhypergraph): hence any set of vertices from $\Q_d$ that all have the same weight is Ramsey. A natural question to ask is: which sets $S$ corresponding to unions of cliques of different weights from $\Q_d$ are Ramsey?

The answer to this question depends on the number of cliques involved. In particular we determine which unions of 2 or 3 cliques are Ramsey and then show, using a probabilistic argument, that any non-trivial union of 39 or more cliques of different weights cannot be Ramsey.

A key tool is a lemma which reduces questions concerning monochromatic configurations in the hypercube to questions about monochromatic translates of sets of integers.

\end{abstract}
\section{Introduction}
Ramsey's theorem is a seminal result of extremal combinatorics. It implies that any $2$-coloring of a sufficiently large $r$-uniform hypergraph will contain a monochromatic copy of a complete subgraph of a given size \cite{Ram}.

The question we wish to address is: what types of monochromatic sets are unavoidable in any 2-coloring of the vertices of a large hypercube? Such sets are said to be \emph{Ramsey}. Since the set of vertices of weight $r$ in a hypercube correspond to a complete $r$-uniform hypergraph it is natural to ask whether  sets of vertices corresponding to unions of complete hypergraphs of different weights can be Ramsey. Our main results show that this can happen for some unions of two or three complete hypergraphs (Theorems \ref{2cliques:thm} and \ref{3cliques:thm}), but that it cannot occur for arbitrarily large unions (Theorem \ref{lll:thm}).

In the next section we give the required definitions and show that when considering which subsets of vertices of the hypercube are Ramsey we may restrict our attention to particularly simple ``layered'' colorings (Theorem \ref{layer:thm}).

As far as we are aware this paper is the first to consider Ramsey problems for the vertices of the hypercube. There is an extensive literature, however, on the corresponding problems for edge-colorings of the hypercube.

Chung \cite{Ch} showed that for all $k\geq 2$ and all $r\geq 1$, there exists $N$ such that if $n \geq N$, every edge-coloring of $\Q_n$ with $r$ colors contains a monochromatic copy of $C_{4k}$. Moreover she gave a 4-coloring of $\Q_n$ with no monochromatic copy of $C_6$, while Conder \cite{Con} found a 3-coloring with this property. 

Alon, Radoi\v{c}i\'{c}, Sudakov, and Vondr\'{a}k \cite{ARSV} extended this to show that for all $k\geq 2$ and all $r\geq 1$, there exists $N$ such that if $n \geq N$, every edge-coloring of $\Q_n$ with $r$ colors contains a monochromatic copy of $C_{4k+2}$. 

Axenovich and Martin \cite{AM} gave a $4$-coloring of the edges of $\Q_n$ containing no induced monochromatic copy of $C_{10}$.

So-called \emph{$d$-polychromatic colorings} have also been considered previously: these are edge colorings of the hypercube with $p$ colors so that every $d$-dimensional subcube contains every color. Alon, Krech and Szab\'o \cite{AKS} give upper and lower bounds for the maximum number of colors for which a $d$-polychromatic colorings exists. Their lower bound was later proved to be exact by Offner \cite{Off1}. They also considered $d$-polychromatic colorings for vertices of the hypercube. Recently Stanton and \"Ozkahya \cite{SO} have also considered some of the questions raised by Alon, Krech and Szab\'o.

Related Tur\'an-type problems for both edges and vertices of the hypercube, were also previously considered. Chung \cite{Ch} gave bounds on the density of edges required to guarantee a copy of $\Q_2$ and this was improved recently by Thomason and Wagner \cite{TW}. Chung also showed that any positive density of edges in a large hypercube guarantees a copy of $C_{4k}$, for $k\geq 2$. More recently this was extended to $C_{4k+2}$ ($k\geq 3)$ by F\"uredi and \"Ozkahya.  For a unified proof of the theorems of Chung, F\"uredi and \"Ozkahya, see Conlon \cite{DC}. 

The first Tur\'an-type result for vertices of the hypercube is due to E.A. Kostochka \cite{K} who showed that any subset of the vertices of the hypercube of density greater than $2/3$ will contain a copy of $\Q_2$. For related results see Johnson and Talbot \cite{JT}.
\section{Definitions and equivalences}
For $a,b\in \mathbb{N}$, $a<b$ we define $[a]=\{1,2,\ldots,a\}$ and $[a,b]=\{a,a+1,\ldots,b\}$.

 For $n\geq 1$ let $V_n=\{0,1\}^n$. The \emph{$n$-dimensional hypercube}, $\Q_n$, is the graph with vertex set $V_n$ and  edges between vertices that differ in exactly one coordinate.

If $1\leq d \leq n $ then an \emph{embedding} of $\Q_d$ into $\Q_n$ is an injective map $\psi :V_d\to V_n$ that preserves the edges of $\Q_d$. Note that the image of $V_d$ under any such embedding consists of $2^{d}$ elements of $V_n$ given by fixing $n-d$ coordinates and allowing the other $d$ coordinates to vary. We refer to the image of such an embedding as a ($d$-dimensional) \emph{subcube} of $\Q_n$. 

\begin{figure}
\begin{center}
\includegraphics[scale=0.35]{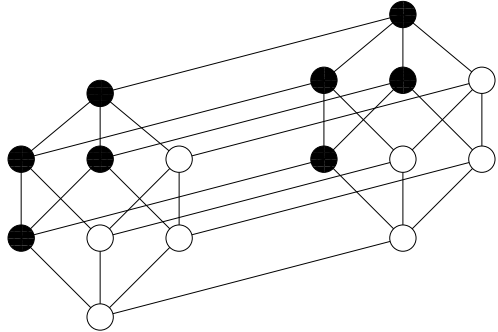}
\end{center}
\caption{A 3-dimensional subcube of $\Q_4$}
\end{figure}
 
Given $F\subseteq \Vd$ and $S\subseteq V_n$, with $1\leq d\leq n$, we say that $S$ \emph{contains a copy of $F$} if there exists an embedding $\psi :\Vd\to V_n$ satisfying $\psi(F)\subseteq S$.

For $t\geq 2$, a \emph{$t$-coloring} of $\Q_n$ is a map $c:V_n\to [t]$. A $t$-coloring of $\Q_n$ \emph{contains a monochromatic copy of $F$} if there is a color $i\in [t]$ such that $c^{-1}(i)$ contains a copy of $F$.

We say that a set $F\subseteq \Vd$ is \emph{$t$-Ramsey} if there exists $n_0(F,t)$ such that for all $n\geq n_0$, every $t$-coloring of $\Q_n$ contains a monochromatic copy of $F$.

For the remainder of this paper we will work with a different model of the hypercube: the \emph{Boolean lattice}, in which vertices of the hypercube  are identified with subsets of $[n]$. To be precise, if $2^{[n]}=\{A:A\subseteq [n]\}$ is the powerset of $[n]$, then the poset $(2^{[n]},\subseteq)$ has $\Q_n$ as its Hasse diagram. We identify $\Vn$ with $2^{[n]}$ via the natural isomorphism $s:V_n\to 2^{[n]}$, $s(x)=\{i:x_i=1\}$. 

We are interested in characterising those subsets of $\Vd$ which are $t$-Ramsey. The simplest example is given by Ramsey's theorem. For $a,t \geq 0$ a \emph{clique} of \emph{order} $t$ and \emph{weight} $a$ is a family consisting of all $a$-sets from a set of size $t$. Given a set $K$ with $|K|=t$ we denote this by $K^{(a)}$.
\begin{theorem}[Ramsey \cite{Ram}]\label{Ramsey:thm} For $t\geq 2$, all cliques are $t$-Ramsey.
\end{theorem}
A trivial corollary is that any family of sets which are all the same size is
$t$-Ramsey for $t\geq 2$.  It is also obvious that any family of sets which
contains members of even and odd weight is not $2$-Ramsey since coloring all sets of even weight red and all sets of odd weight blue avoids monochromatic copies of such a family.

For $A\in \Vd$ the \emph{weight} of $A$ is $|A|$.  The collection of all sets of a fixed weight in $\Vd$ gives a special type of clique, called a \emph{layer}. The layer containing all sets of weight $i$ from $\Q_n$ is called the \emph{$i^{\textrm{th}}$ layer} (of $\Q_n$) and we denote it by $L_i$.
\begin{figure}
\begin{center}
\includegraphics[scale=0.35]{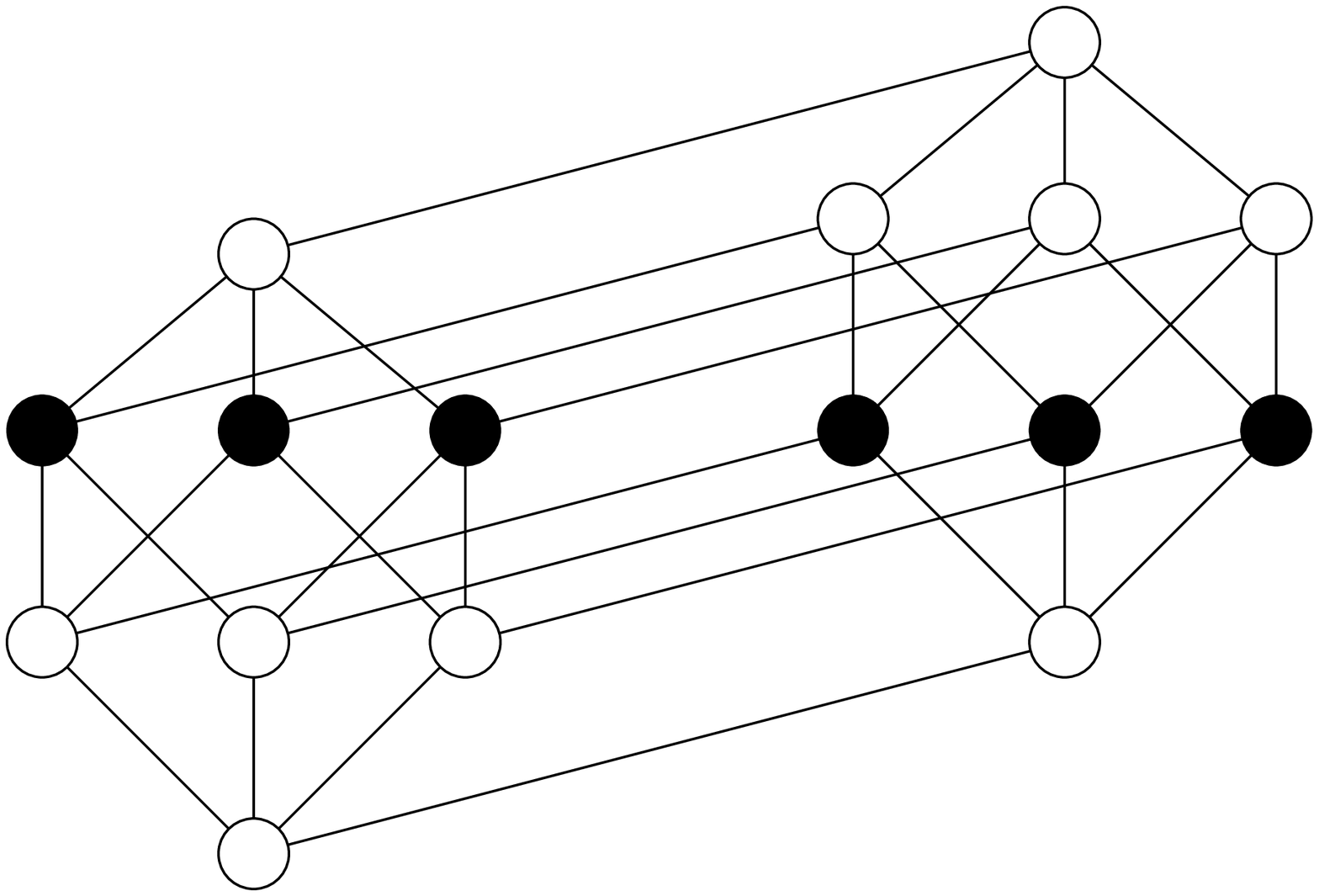}
\end{center}
\caption{The $2^{\tr{nd}}$ layer of $\Q_4$}
\end{figure}

\begin{figure}
\begin{center}
\includegraphics[scale=0.35]{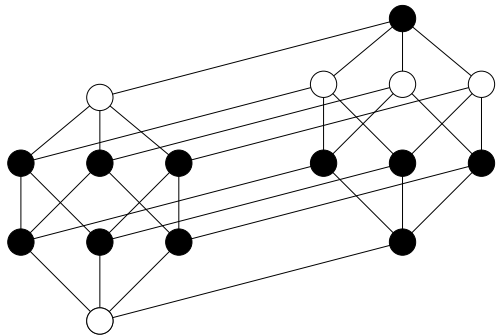}
\end{center}
\caption{A layered 2-coloring of $\Q_4$ with no monochromatic copy of $\Q_2$.}
\end{figure}
A particularly simple $t$-coloring of $\Q_n$ is one that is constant on each layer. We call such a coloring \emph{layered}. A set $F\subseteq \Vd$ is \emph{$t$-layer-Ramsey} if  there exists $n_L(F,t)$ such that for all $n\geq n_L$, every layered $t$-coloring of $\Q_n$ contains a monochromatic copy of $F$.

Our first result says that there is no difference between $t$-Ramsey and $t$-layer-Ramsey sets.
\begin{theorem}\label{layer:thm}
A set $S \subseteq \Vd$ is $t$-Ramsey iff it is $t$-layer-Ramsey.\end{theorem}

For the non-trivial implication in Theorem \ref{layer:thm} we require the following lemma. 
\begin{lemma}\label{layer:lem}
If $s,t\geq 1$ then there exists $c_L(s,t)$ such that any $t$-coloring of $\Q_n$, where $n\geq c_L$, contains a copy of $\Q_s$ such that the restriction of the coloring to $\Q_s$ is layered.\end{lemma}
\begin{proof}  By Ramsey's theorem, for any $s\geq l\geq 0$ and $t\geq 2$ there exists an integer $R(s,l,t)$ such that whenever the collection of all $l$-sets from $[R(s,l,t)]$ are $t$-colored there is a monochromatic clique of order $s$. We define a sequence  $f_0,f_1,...,f_{s-1}$ by: $f_0 = s$, $f_i = R(f_{i-1},i,t)$  for $i>0$.

We claim that $c_L(s,t) = f_{s-1}$ will suffice.  Suppose that $\chi$ is a $t$-coloring of $\Q_{f_{s-1}}$. By the definition of the $\{f_i\}$ there exists a nested sequence of sets $F_0 \subseteq F_1 \subseteq F_2 \cdots \subseteq F_{s-1} = [c_L(s,t)]$ such that $|F_j|=f_j$ for $j=0,1,\ldots,s-1$ and the restriction of $\chi$ to $F_{j-1}^{(j)}$ is monochromatic for $j=1,2,...,s-1$.  (To see this start with $F_{s-1}$ and work down.)  
Hence the restriction of $\chi$ to $F_0^{(j)}$ is monochromatic for $j=1,2,\ldots,s-1$.  Adding the
empty set and $F_0$ then gives the desired copy of $\Q_s$ on which the restriction of $\chi$ is layered. \end{proof}

%
%

We remark that our proof  actually implies that in any $t$-coloring of $\Q_{n}$, where $n\geq c_L(s,t)$, and for any $B\in \Vn$ there is copy of $\Q_s$ with ``$B$ at the bottom'' for which the restriction of the coloring is layered. The integer $c_L(s,t)$ produced by this ``tower of Ramsey numbers'' is obviously rather large if $s$ is large.  It would
be interesting to find a good upper bound for the smallest possible value
of $c_L(s,t)$.
\begin{proof}[Proof of Theorem \ref{layer:thm}]
Since a layered $t$-coloring of the cube is still a $t$-coloring one implication is trivial. 

For the converse suppose that $S\subseteq \Vd$ is $t$-layer-Ramsey. Let $\chi$ be a $t$-coloring of $\Q_n$ with $n\geq c_L(n_L(S,t),t)$. By Lemma \ref{layer:lem} there is subcube $\Q_{n_L(S,t)}$ of $\Q_n$ such that the restriction of $\chi$ to this subcube is layered. Since $S$ is $t$-layer-Ramsey this subcube contains a monochromatic copy of $S$.
\end{proof}

For $S\subseteq \Vd$ we define $W_d(S)=\{|A|:A\in S\}$. For example, if \[
S=\{\{1,2,3\},\{1,2,4\},\{1,3,4\},\{5\}\}\subseteq V_5\] then $W_5(S)=\{1,3\}$. 

A layered $t$-coloring $c$ of $\Q_n$ is equivalent to a $t$-coloring $\hat{c}$ of the integers $\{0,1,\ldots,n\}$ (given by $\hat{c}(i)=c(L_i)$). Thus, using Theorem \ref{layer:thm}, we can translate our original question ``which subsets of the hypercube are $t$-Ramsey'', into a question concerning $t$-colorings of the integers that avoid certain distance sets. This is the key observation which underlies most of our results.

For $D\subseteq \mathbb{Z}$ and $b\in \mathbb{Z}$ we define $D+b=\{d+b\mid d\in D\}$ to be the translation of $D$ by $b$. We say that a family of sets of integers $\D=\{D_1,D_2,\ldots, D_k\}$ is \emph{$t$-translate-Ramsey} if for every $t$-coloring of the integers, $c:\mathbb{Z}\to [t]$ there exists $D\in \D$ and  $j\in \mathbb{Z}$ such that $D+j$ is monochromatic, i.e. every $t$-coloring of the integers contains a monochromatic translate of a set from the family.

For example the family $\D = \{\{0,1\},\{0,2\},\ldots,\{0,t\}\}$ is $t$-translate-Ramsey but not $(t+1)$-translate-Ramsey (to get a $(t+1)$-coloring with no monochromatic translation of any set in $\D$ just repeat a list of the $t+1$ distinct colors).
\begin{lemma}\label{fin:lemma}
If  $\D=\{D_1,D_2,\ldots, D_k\}$ is $t$-translate-Ramsey, then there exists $n_T(\D,t)$ such that every $t$-coloring of $[n_T(\D,t)]$ contains a monochromatic translate of a set from $\D$.
\end{lemma}
\begin{proof} (This follows easily by compactness but for completeness we give an elementary self-contained proof.)
Suppose that $\D$ is $t$-translate-Ramsey. Let $d=\max\{\max D - \min D: D\in\D\}$. We will show that $n_T(\D,t)=d(t^d+1)$ will suffice. Suppose, for a contradiction, that there is a $t$-coloring $c:[d(t^d+1)]\to [t]$ with no monochromatic translate of any $D\in\D$. Let $B_i=\{1+(i-1)d,2+(i-1)d,\ldots,id\}$ then $[n_T]=B_1\cup B_2\cup\cdots B_{t^d+1}$. Since there are $t^d+1$ blocks $B_i$ and each block contains $d$ integers,  there exist blocks $B_i,B_j$ with $i<j$ such that $B_i$ and $B_j$ are colored identically. Now color the integers periodically using the colors of $B_i,B_{i+1},\ldots,B_{j-1}$

If this coloring of $\mathbb{Z}$ contains a monochromatic translate of $D\in \D$ then by definition of $d$ this translate meets at most two consecutive blocks of the coloring. Moreover since the coloring of $[n_T]$ contained no such monochromatic translate it must meet two consecutive blocks which did not occur in the original coloring of $[n_T]$. But no such pair of blocks occur (since $B_j$ and $B_i$ are colored identically).
\end{proof}
It would be interesting to find an order of magnitude estimate for the
smallest possible $n_T(\D,t)$.  We note that the proof of Lemma \ref{fin:lemma} also shows that if $\D$ is not $t$-translate-Ramsey then there exists a periodic coloring of the integers with period of length at most $dt^d$ which contains no monochromatic translates of sets from $\D$.

The link between $t$-Ramsey subsets of vertices of the hypercube and $t$-translate-Ramsey families is given by considering which collections of layers a given subset $S\subseteq \Vd$ can meet in the hypercube under all possible embeddings. 

For $S\subseteq \Vd$ we define \[
W_d^*(S)=\{W_d(\psi(S)):\psi :\Vd\to \Vd,\textrm{ is an automorphism}\}.\]
Any automorphism of $\Q_d$ can be expressed (in the Boolean lattice model) as a set complement followed by a permutation of $[d]$. Since a permutation of the $d$ labels does not alter the weight of $v\in V(\Q_d)$ we can restrict our attention to the \emph{simple automorphisms} of $\Q_d$ of the form $\psi_B(A)=A\Delta B$, $A,B\in 2^{[d]}$ when determining $W_d^*(S)$:
\[
W_d^*(S)=\{W_d(\psi(S)):\psi :\Vd\to \Vd,\textrm{ is a simple automorphism}\}.\] 

For example, let $S=\{\{1,2,3\},\{1,2,4\},\{1,3,4\},\{2,3,4\},\{5\}\} \subseteq V_6$.  Now $W_6(S)=\{1,3\}$, while 
\begin{multline*}
W_6^*(S)=\{\{1,3\},\{2,4\},\{3,5\},\{0,4\},\{1,5\},\{2,6\},
\{0,2,4\},\{1,3,5\},\{2,4,6\}\}.\end{multline*}
Note that if $d_1\leq d_2$ and $S$ can be embedded in $\Q_{d_1}$ then $S$ can also be embedded in $\Q_{d_2}$, so $W^*_d(S)$ depends on the value of $d$. (In our example above $W_5^*(S)=\{\{1,3\},\{2,4\},\{0,4\},\{1,5\},\{0,2,4\},\{1,3,5\}\}$.) 
In general if $d_1\leq d_2$ then $W^*_{d_1}(S)\subseteq W_{d_2}^*(S)$, while $W_{d_2}^*(S)\setminus W_{d_1}^*(S)$ consists of translates (in $[d_2]$) of sets from $W_{d_1}^*(S)$. For this reason $W^*_{d_1}(S)$ is $t$-translate-Ramsey iff $W^*_{d_2}(S)$ is $t$-translate-Ramsey and so we define $W^*(S)$ to be $W_d^*(S)$, with $d$ minimal such that $S\subseteq \Vd$.

When considering whether or not $W^*(S)$ is $t$-translate-Ramsey it is natural to define $W'(S)$ to be the family of all translates of sets from $W^*(S)$ which have smallest element zero and which are minimal with respect to inclusion. 
Thus in our example above  we have $W'(S)=\{\{0,2\},\{0,4\}\}$.

\begin{lemma}\label{translate:lem} If $S\subseteq \Vd$ then the following are equivalent.
\begin{itemize}
\item[(i)] $S$ is $t$-Ramsey;
\item[(ii)]  $W^*(S)$ is $t$-translate-Ramsey;
\item[(iii)] $W'(S)$ is $t$-translate-Ramsey.
\end{itemize}
\end{lemma}
\begin{proof} Clearly $W^*(S)$ is $t$-translate-Ramsey iff $W'(S)$ is $t$-translate-Ramsey (taking translations and removing supersets can have no effect on whether or not a family is $t$-translate-Ramsey). We will show that $S\subseteq \Vd$ is $t$-Ramsey iff  $W^*(S)$ is $t$-translate-Ramsey.

Suppose that $S\subseteq \Vd$ is $t$-Ramsey and $n_0(S)$ is sufficiently large that any $t$-coloring of $\Q_{n_0}$ contains a monochromatic copy of $S$. Now take a $t$-coloring $c$ of $\mathbb{Z}$. This induces a layered coloring of $\Q_{n_0}$ given by $\hat{c}(L_i)=c(i)$. By definition of $n_0$ there is a subcube of $Q_{n_0}$ containing a monochromatic copy of $S$. The set of layers of $\Q_{n_0}$ in which this copy of $S$ lies is a translate of some $D\in W^*(S)$ and hence there is a monochromatic translate of $D$ in the original coloring of the integers. Hence $W^*(S)$ is $t$-translate-Ramsey.

Conversely, suppose that $W^*(S)$ is $t$-translate-Ramsey. Lemma \ref{fin:lemma} implies that there exists $n_0$ such that any $t$-coloring of $[n_0]$ contains a monochromatic translate of some $D\in W^*(S)$. Let $c$ be a layered $t$-coloring of $Q_{n_0}$. Define a  $t$-coloring $\hat{c}$ of $[n_0]$ by $\hat{c}(i)=c(L_i)$. By definition of $n_0$, this coloring contains a monochromatic translate of some $D\in W^*(S)$. Hence there is a subcube of $Q_{n_0}$ containing a monochromatic copy of $S$. So $S$ is $t$-layer-Ramsey and hence by Theorem \ref{layer:thm} is $t$-Ramsey.
\end{proof}

Given Ramsey's theorem (Theorem \ref{Ramsey:thm}), telling us that all cliques are $t$-Ramsey for all $t\geq 2$, a natural question is to ask whether unions of cliques can also be $t$-Ramsey. The answer, rather surprisingly, depends on how many cliques we have. 
\section{Unions of cliques}
\subsection{Preliminaries}
In order to decide which unions of cliques are Ramsey we need to consider the different sets of layers in which such unions may be embedded. 

Recall that a clique of weight $a$ and order $s$ consists of all $a$-sets from a vertex set of size $s$. (Note that here we use the term \emph{vertex} to mean a vertex of a hypergraph, rather than a vertex of the hypercube.) We say that a union of cliques is \emph{vertex disjoint} if the vertex sets of distinct cliques are pairwise disjoint. For example if $S_1=[3]$, $S_{2}=[4,10]$ and $S_3=[13,20]$ then $S=S_1^{(2)}\cup S_2^{(4)}\cup S_3^{(5)}$ is a vertex disjoint union of cliques. We will focus mainly on vertex disjoint unions due to the following simple result.

\begin{lemma}\label{nondisj:lem}
If $t\geq 2$ and $S$ 
 is a vertex disjoint union of cliques that is \emph{not} $t$-Ramsey then any union of cliques with the same weights and orders as $S$ (but not necessarily vertex disjoint) is also not $t$-Ramsey.
\end{lemma}

Let $S=K_1\cup K_2\cdots \cup K_s$ be a 
union of cliques from $\Q_d$. Suppose that $K_i$ is of weight $a_i$ and order $a_i+t_i$, i.e. $K_i\simeq K_{a_i+t_i}^{(a_i)}$. We wish to determine $W^*(S)$. Consider a simple automorphism given by $\psi_B(A)=A\Delta B$ for some $B\subseteq [d]$. 
Let us suppose that $|B\cap V(K_i)|=b_i$ and $|B|=b$. If $b_i\in \{0,a_i+t_i\}$ then $\psi_B(K_i)$ will be contained in a single layer (either $b+a_i$ or $b-a_i$). However if $0<b_i<a_i+t_i$ then $\psi_B(K_i)$ will meet multiple layers. In order to succinctly describe which layers $\psi_B(K_i)$ will meet  we need the following notation.

For integers $x<y$ of the same parity we define  \[
[x,y]_2=\{x,x+2,\ldots,y-2,y\}.\]

\begin{lemma}\label{embed:lem}
If $S,\psi_B, b_1,\ldots,b_s$ and $b$ are as above and $D_i$ 
 denotes the set of layers that $\psi_B(K_i)$ meets then
\[D_i=[\max\{a_i-2b_i,-a_i\},\min\{a_i,a_i+2(t_i-b_i)\}]_2+b.\] Moreover precisely one of the following holds for each $1\leq i\leq s$:
\begin{itemize}
\item[(i)] $-a_i+b$ or $a_i+b$ belongs to $D_i$,
\item[(ii)] $t_i<b_i<a_i$ and $D_i=[a_i-2b_i,a_i-2(b_i-t_i)]_2+b$.
\end{itemize}
\end{lemma}
\begin{proof} The first part follows by checking how large $|A\Delta B|$ can be as   $A$ varies over the  sets from  $K_i\simeq K_{a_i+t_i}^{(a_i)}$, where $|B|=b$ and $|A\cap V(K_i)|=b_i$. 

For the second part suppose that (i) fails to hold. Now $-a_i+b\not\in D_i$ implies that $a_i-2b_i>-a_i$, and hence  $b_i<a_i$. Similarly since $a_i+b\not\in D_i$ we must have $t_i<b_i$. Hence  $D_i=[a_i-2b_i,a_i-2(b_i-t_i)]_2+b$.
\end{proof}

Given $S, \psi_B, b_1,\ldots,b_s,$ and $b$ as above define 
\[
E(b;b_1,b_2,\ldots,b_s)=\bigcup_{i=1}^s[\max\{a_i-2b_i,-a_i\},\min\{a_i,a_i+2(t_i-b_i)\}]_2+b.\] 
Since $\psi_B(S)=\psi_B(K_1)\cup\cdots \cup \psi_B(K_s)$, $\psi_B(S)$ will meet precisely those layers contained in $E(b;b_1,\ldots,b_s)$.

Thus the family of all possible sets of layers occupied by embeddings of $S$  depends on which values of $b$ and $b_1,\ldots,b_s$ can occur:
\[
W^*(S)=\{E(b;b_1,\ldots,b_s)\mid \exists B\subseteq [d], |B|=b, |B\cap K_i|=b_i, 1\leq i \le s\}.\]
Clearly each $b_i$ must satisfy $0\leq b_i \leq a_i+t_i$ and if the cliques in $S$ are vertex disjoint then all such values are possible. If, however, two cliques overlap, say $|V(K_i)\cap V(K_j)|=c\geq 1$, then $b_i\geq a_i+t_i-d \implies b_j \geq c-d$, so for example $b_i=a_i+t_i$ and $b_j=0$ is impossible. 

We can now prove Lemma \ref{nondisj:lem}.
\begin{proof}[Proof of Lemma \ref{nondisj:lem}]
Let $S$ be a vertex disjoint union of cliques that is not $t$-Ramsey. If $\hat{S}$ is any union of cliques with the same weights and orders as those in $S$ then by the above discussion we have $W'(\hat{S})\subseteq W'({S})$ (any choice of $b_1,\ldots b_s$ that can occur for an embedding of $\hat{S}$ can also occur for an embedding of ${S}$). Now if $\hat{S}$ is $t$-Ramsey then Lemma \ref{translate:lem} implies that $W'(\hat{S})$ is $t$-translate-Ramsey. But then $W'(\hat{S})\subset W'({S})$ so $W'(S)$ is $t$-translate-Ramsey and so $S$ is $t$-Ramsey, a contradiction.
\end{proof}
For the remainder of this section we will restrict attention to the case that $S$ is a  vertex disjoint unions of cliques. Note that in this case for any embedding we have $b=\sum_{i=1}^sb_i$, so we write $E(b_1,\ldots,b_s)$ for $E(b;b_1,\ldots,b_s)$.

Embeddings of $S$ in which $b_i\in \{0,a_i+t_i\}$ for each $1\leq i \leq s$ will play a special role and we call these \emph{principal embeddings} of $S$. We define
\[
P^*(S)=\{{E}(b_1,\ldots,b_s)\mid b_i\in \{0,a_i+t_i\}, 1\leq i \leq s\},\]
to denote those sets in $W^*(S)$ achieved by principal embeddings. Note that all $E\in P^*(S)$ are translates of sets of the form $\{x_1a_1,x_2a_2,\ldots , x_sa_s\}$, for some choice of signs $x_1,\ldots,x_s\in \{-1,+1\}$.

For example consider $S_1=[6]^{(4)}\cup [7,15]^{(8)}$. In this case the principal embeddings yield
\[
P^*(S_1)=\{\{4,8\}, \{2,14\}, \{1,13\}, \{7,11\}\}.\]  
We will let $P'(S)$ denote those sets from $W'(S)$ which are translates of sets from $P^*(S)$. So in this example we have \[
P'(S_1) =\{\{0,4\},\{0,12\}\}.\]
Note that a coloring $c$ of $\mathbb{Z}$ which alternates colors on the integers in each congruence class modulo 4 contains no monochromatic translate of either set in $P'(S_1)$.  However, while this coloring also contains no monochromatic translate of the set $\{4,12,14\}$ produced by the non-principal embedding obtained by taking $b_1 = 6$ and $b_2 = 2$, it does contain a monochromatic translate of the set $\{6,8\}$ produced by the non-principal embedding obtained by taking $b_1 = 0$ and $b_2 = 2$. Since $\{0,2\}$ and $\{0,4\}$ are both in $W'(S_1)$, $S_1$ is $2$-Ramsey. On the other hand, if $S_2=[6]^{(4)}\cup[7,16]^{(8)}$ then $P'(S_2)=P'(S_1)$ yet, as Theorem \ref{2cliques:thm} will show, $S_2$ is not $2$-Ramsey ($\{0,2\}$ is not in $W'(S_2)$).

Our next result tells us that if the sizes of the cliques are not too small compared to their weights we need only consider principal embeddings.
\begin{prop}\label{7:prop} If $S$ is as in Lemma \ref{embed:lem} with $t_i\geq a_i-1$ for each $i$ then $W'(S)=P'(S)$.
\end{prop}
\begin{proof}
No $b_i$ can satisfy the inequality in Lemma \ref{embed:lem} (ii). Hence each set in $W'(S)$ contains a set in $P'(S)$, so in fact must equal a set in $P'(S)$.\end{proof}

\subsection{Two cliques}
For integers $a,b,c$ we denote ``$a$ is congruent to $b$ modulo $c$'' by $a\equiv_{c} b$. We extend this in the obvious way to sets: e.g. $\{8,14\}\equiv_4 \{0,2\}$.

\begin{lemma}\label{2cliquesa:lem}If $S=K_1\cup K_2$ is the vertex disjoint union of two cliques of weights $a_1$ and $a_2$ and orders $a_1+t_1$ and $a_2+t_2$ respectively, with $a_1 = p_12^{r_1}$ and $a_2 = p_22^{r_2}$ where $t_1,t_2,r_1,r_2\geq 0$, $r_1\leq r_2$, and $p_1$, $p_2$ are odd then
\begin{itemize}
\item[(a)]  the reduced family of sets of layers of principal embeddings is
\[
P'(S)=\{\{0, |a_1 - a_2|\}, \{0, a_1 + a_2\}\}.
\]
\item[(b)] If $r_1 = r_2$ then S is $2$-Ramsey.
\item[(c)]  If $r_1 < r_2$ and $c$ is a $2$-coloring of $\mathbb{Z}$, then there is no monochromatic translate of either set in $P'(S)$ iff $c(x)\neq c(y)$ for all $x,y$ such that $|x-y| = d2^{r_1}$, where $d=\gcd(p_1,p_2)$.
\end{itemize}
\end{lemma}
\begin{proof}
By definition
\[
P^*(S)=\{\{a_1,a_2\}, \{t_1, a_1 + a_2 + t_1\}, \{a_1 + a_2 + t_2, t_2\}, \{a_2 + t_1 + t_2,a_1 + t_1 + t_2\}\}\]
so (a) follows immediately.

(b) If $r_1=r_2$ then, since $\{p_1+p_2,p_1-p_2\}\equiv_4 \{0,2\}$,  the integer \[
\frac{(a_1-a_2)(a_1+a_2)}{2^{r_1+1}}=\frac{p_1-p_2}{2}(a_1+a_2)=\frac{(p_1+p_2)}{2}(a_1-a_2)\] is an odd multiple of one of $a_1+a_2$ and $|a_1-a_2|$, and an even multiple of the other. 
Hence any $2$-coloring of $\mathbb{Z}$ must contain a monochromatic translate of one of the sets in $P'(S)$.

(c) If $r_1 < r_2$ then both $|a_1-a_2|$ and $a_1+a_2$ are odd multiples of $d2^{r_1}$.  Now if $c(x)\neq c(y)$ for all $x,y$ such that $|x-y|=d2^{r_1}$ then $c(x)\neq c(y)$ for all $x,y$ such that $|x-y|$ is an odd multiple of $d2^{r_1}$. Hence there is no monochromatic translate of either set in $P'(S)$. Conversely, suppose $c(x)=c(y)$ for some $x,y$ with $x-y=d2^{r_1}$.  If $r_1 < r_2$ then $\gcd(|a_1-a_2|, a_1+a_2) = d2^{r_1}$, so there exist integers $k$ and $m$, one even and one odd, such that $(a_1+a_2)k - (a_1-a_2)m = d2^{r_1}$. By symmetry we may suppose $k$ is even and $m$ is odd. Now if $z=x+(a_1-a_2)m = y+(a_1+a_2)k$ then $|z-x|$ is an odd multiple of $|a_1-a_2|$, and  $|z-y|$ is an even multiple of $a_1+a_2$. Since $c(x)=c(y)$  there must be a monochromatic translate of a set from $P'(S)$.
\end{proof}

\begin{lemma}\label{2cliquesb:lem} For each positive integer $m$ divisible by $4$, there exists a $2$-coloring $c$ of $\mathbb{Z}$ such that $c(x)\neq c(y)$ for all $x,y$ with $|x-y| = m$, and $c(z)=c(z+2)=c(z+4)$ does not occur for any $z$.\end{lemma}
\begin{proof}  The period $2m$ coloring obtained by taking $RRBBRRBB\ldots RRBB$ on $[0,m-1]$, then taking the complement of these colors on $[m,2m-1]$, and so on, satisfies the required properties.
\end{proof}
\begin{theorem}\label{2cliques:thm}  Let $S=K_1\cup K_2$ be the vertex disjoint union of two cliques of weights $a_1$ and $a_2$ and orders $a_1+t_1$ and $a_2+t_2$ respectively, with $a_1 = p_12^{r_1}$ and $a_2 = p_22^{r_2}$ where $t_1,t_2,r_1,r_2\geq 0$, $p_1$, $p_2$ are odd integers, and $r_1\leq r_2$. Then S is $2$-Ramsey iff at least one of the following is satisfied
\begin{itemize}
\item[(1)] $r_1=r_2$;
\item[(2)] at least one of $t_1$ or $t_2$ is equal to 0, and $a_1$ and $a_2$ are both even;
\item[(3)] $t_1$ or $t_2$ is equal to $1$, and $2 \leq r_1 < r_2$.
\end{itemize}
\end{theorem}
\begin{proof}
If (1) is satisfied then $S$ is $2$-Ramsey by Lemma \ref{2cliquesa:lem} (b).  Assume $r_1<r_2$ and that (2) is satisfied, say with $t_1=0$ and $a_1 < a_2$.  The sets $\{0, a_2 - a_1 +2\}$ (by taking $b = b_1 = 1$) and $\{0, a_2 - a_1\}$ are both in $W'(S)$.  If $x$ is any integer such that $c(x-2)$ is not equal to $c(x)$, then $x + a_2 - a_1$ has the same color as $x$ or $x-2$, so there is a monochromatic translate of a set in $W'(S)$.  The argument is virtually the same if $t_2=0$ or $a_2<a_1$.
 
Now suppose (3) is satisfied.  If $c$ is a $2$-coloring of $\mathbb{Z}$ with no monochromatic translate of either set in $P'(S)$, then it must have the form prescribed in Lemma \ref{2cliquesa:lem} (c), so there exist integers $x$ and $y$ with opposite colors such that $y-x$ is a positive multiple of $4$.  That means there exists an integer $z \in [x,y]_2$ such that $c(z)=c(z+2)$.  Now there are four cases.
 
If $a_2 < a_1$ and $t_1=1$ then the set $\{(a_1 + a_2)/2, (a_1 + a_2)/2 + 2\}$ is in $W^*(S)$ (take $b = b_1 = |a_1 - a_2|/2$), and there is a monochromatic translate of this set. An identical argument works if $a_1<a_2$ and $t_2=1$.
 
If $a_1 < a_2$ and $t_1=1$ then for each $i \in [a_1]$, a translate of the set $A_i = \{0, 2, a_2 - a_1 +2i\}$ is in $W^*(S)$ (take $b = b_1 = i$).  This means that if $c(0)=c(2)=R$, to avoid a red translate of some set $A_i \in W^*(S)$, all of the integers in $[a_2 - a_1 + 2, a_2 + a_1]_2$ must be blue.  For each consecutive pair of blue integers in this set of size $a_1$, to avoid a blue translate of some set $A_i$, there must be a set of $a_1$ consecutive red even integers.  Taking their union forces every integer in $[2(a_2 - a_1 +2), 2(a_2 + a_1-1)]_2$ to be red.  Thus at this second stage we have $2a_1-2$ consecutive red integers of the same parity. Continuing this process, at the $k$th stage there must be $k(a_1-2)+2$ consecutive integers of the same parity with the same color. Since $a_1\geq 4$, this cannot be true for large $k$. An identical argument works if $a_2<a_1$ and $t_2=1$. Hence $S$ is $2$-Ramsey.

Conversely, assume that $S$ does not satisfy (1), (2), or (3).  If $a_1$ and $a_2$ have different parities then every member of $W^*(S)$ contains numbers of different parity, and thus $S$ is not $2$-Ramsey. So we can assume $1 \leq r_1 < r_2$.  If $r_1 = 1$ then we take a coloring which alternates colors on the even integers and on the odd integers.  Since both $a_1 - a_2$ and $a_1 + a_2$ are odd multiples of $2$, there is no monochromatic translate of either set in $P'(S)$.  Since both $t_1$ and $t_2$ are positive, each non-principal embedding contains two integers whose difference is $2$, so these cannot be monochromatic.  Hence $S$ is not $2$-Ramsey.  
 
Now assume that $2\leq r_1<r_2$ and so $t_1,t_2\geq 2$.  By Lemma \ref{2cliquesb:lem} there exists a coloring $c$ of the type prescribed in Lemma \ref{2cliquesa:lem} (c) such that $c(z)=c(z+2)=c(z+4)$ does not occur for any $z$.  Since $t_1,t_2\geq 2$, any set in $W'(S)\setminus P'(S)$ contains a translate of $\{0,2,4\}$, so there is no monochromatic translate of such a set.  Furthermore, since $a_1 - a_2$ and $a_1 + a_2$ are both odd multiples of $d2^{r_1}$, there is no monochromatic translate of either set in $P'(S)$.  Hence $S$ is not $2$-Ramsey.      
\end{proof}

If $S=K_{a_1+t_1}^{(a_1)}\cup K_{a_2+t_2}^{(a_2)}$ is the union of two cliques which are not vertex disjoint then Lemma \ref{nondisj:lem} tells us that for $S$ to be $2$-Ramsey it must have the same parameters as a vertex disjoint union of cliques that is $2$-Ramsey. In fact we can say more and state the following theorem without proof.
\begin{theorem}
If $S=K_{a_1+t_1}^{(a_1)}\cup K_{a_2+t_2}^{(a_2)}$ is the union of two cliques whose vertex sets overlap in $c\geq 1$ points, $a_1>a_2$, $t_1,t_2\geq 2$ then 
\begin{itemize}
\item[(i)] If $c\geq 3$ then $S$ is not $2$-Ramsey.
\item[(ii)] If $c=2$ then $S$ is $2$-Ramsey iff there exists a positive integer $m$ such that $a_1-a_2=4m$ and $a_1+a_2\equiv 2$ mod $8m$.
\item[(iii)] If $c=1$ then $S$ is $2$-Ramsey iff there exists a positive integer $m$ such that $a_1-a_2=4m$ and $a_1+a_2 \equiv 0,2,4m-2$ or $4m+4$ mod $8m$, or there is an even integer $m$ such that $a_1-a_2=4m$ and $a_1+a_2\equiv 6$ or $8m-4$ mod $8m$.
\end{itemize}
\end{theorem}
\subsection{Three cliques}
If the disjoint union of $s$ cliques of different weights is $t$-Ramsey, then
clearly the disjoint union of any $s'$ of them, for any $s'<s$, is $t$-Ramsey as
well.  The converse obviously does not hold in general, so the following
result is rather surprising.
\begin{theorem}\label{3cliques:thm}
A vertex disjoint union of three cliques of pairwise distinct weights is $2$-Ramsey iff the union of each pair of the cliques is $2$-Ramsey.
\end{theorem}

Due to the various possibilities for the structure of each pair of two of
the three cliques (Theorem \ref{2cliques:thm} (1),(2),(3)), a complete proof of Theorem \ref{3cliques:thm} would
be long.  The main idea of our proof is to assume that the union of each
pair of two of the three cliques in $S$ is $2$-Ramsey, and then show that the
only possible coloring of the integers with no monochromatic translate of
any set in $W'(S)$ is periodic, with a short period.  It is then easy to
show that no such coloring exists.

\begin{lemma}\label{3cliques2:lem}
 Let $a_1,a_2,a_3$ be integers with $a_1>a_2>a_3$ such that $a_2$ and $a_3$
have the same number of factors of $2$ in their prime factorizations.  Let $t_1, t_2, t_3$ be nonnegative integers, and let $S=K^{(a_1)}_{a_1+t_1}\cup K^{(a_2)}_{a_2+t_2}\cup K^{(a_3)}_{a_3+t_3}$ be a vertex disjoint union of cliques.  Any $2$-coloring of the
integers with no monochromatic translate of any set in $W'(S)$ is periodic with
period $2a_1$.
\end{lemma}
\begin{proof}
Let $C$ be any $2$-coloring of the integers with no monochromatic translate of any
set in $W'(S)$. Let $e= \gcd(a_1,a_2,a_3)$ and for each $0\leq i \leq 2e-1$ let $Z_i$ be the set of all integers congruent to $i$ mod $2e$. By the proof of Lemma \ref{2cliquesa:lem}(b) there must be two integers in $Z_0$ with the same color
whose difference is $a_2-a_3$ or $a_2+a_3$.  Assume it is the former, say
$C(0)=C(a_2-a_3)=R$.  Then $C(a_1-a_3)=C(a_1+a_2)=B$ to avoid red translates of $\{a_3,a_2,a_1\}$ and $\{-a_2,-a_3,a_1\}$ respectively. (Note that $a_1-a_3$ and $a_1+a_2$ have the same color and their difference is $a_2+a_3$. If we had instead assumed two integers with difference $a_2+a_3$ are both $R$, then two integers with difference $a_2-a_3$ would be $B$.)  Then $C(2a_1)=C(2a_1+a_2-a_3)=R$ (to avoid blue
translates of $\{-a_3,a_2,a_1\}$ and $\{-a_2,a_3,a_1\}$ respectively).  Continuing,
$C(3a_1-a_3)=C(3a_1+a_2)=B$, and so on (in both directions) so that all integers
congruent to $0$ or $a_2-a_3$ $\mod 2a_1$ are colored $R$, and all integers congruent to $a_1-a_3$ or
$a_1+a_2$ $\mod 2a_1$ are colored $B$. 

Now consider any integer $m$ colored $R$ by the above argument.  Then every integer congruent to $m$ mod $2a_1$ is colored $R$, and if
$C(m + a_2 - a_3) = R$ then, by the same argument as above every integer congruent to $m+a_2-a_3\mod 2a_1$ is also colored $R$. This in turn implies that if $C(m+a_2-a_3)=B $  then all integers congruent to $m+a_2-a_3$ $\mod 2a_1$  are colored $B$ (since if any of them were red they would all be red). Thus all integers congruent to $m+a_2-a_3\mod 2a_1$ have the same color. Similarly all integers congruent to $m+a_2+a_3\mod 2a_1$ must have the same color. Continuing in this way we see that for any fixed integers $x,y$ the set of integers congruent to $x(a_2-a_3)+y(a_2+a_3)\mod 2a_1$ all have the same color (of course for some values of $x$ and $y$ the color is $B$, for others it is $R$). In particular if $d=\gcd(a_2-a_3,a_2+a_3)$ and $j$ is any fixed integer then all integers congruent to $jd \mod 2a_1$ have the same color (and all these integers are in $Z_0$).

 Now $d=\gcd(a_2+a_3,a_2-a_3) = 2\gcd(a_2,a_3)$,  so $2e=\gcd(2a_1,d)$. Hence for each fixed integer
$j$, all integers congruent to $2je$ $\mod 2a_1$ have the same color.   So we have
shown that the coloring $C$ is periodic with period $2a_1$ on $Z_0$.  The same argument can be applied to $Z_i$ for each $1\leq i \leq 2e-1$, showing that $C$ has period $2a_1$ on the integers.
\end{proof}

\begin{lemma}\label{3cliques3:lem} 
Let $a_1,a_2,a_3$ be integers with $a_1>a_2>a_3$ such that $a_1$ and $a_2$
have the same number of factors of $2$ in their prime factorizations.  Let $t_1, t_2, t_3$ be nonnegative integers, and let $S=K^{(a_1)}_{a_1+t_1}\cup K^{(a_2)}_{a_2+t_2}\cup K^{(a_3)}_{a_3+t_3}$ be a vertex disjoint union of cliques.  Any $2$-coloring of the
integers with no monochromatic translate of any set in $W'(S)$ is periodic with
period $2a_3$.
\end{lemma}
\begin{lemma}\label{3cliques4:lem} 
Let $a_1,a_2,a_3$ be integers with $a_1>a_2>a_3$ such that $a_1$ and $a_3$
have the same number of factors of $2$ in their prime factorizations.  Let $t_1, t_2, t_3$ be nonnegative integers, and let $S=K^{(a_1)}_{a_1+t_1}\cup K^{(a_2)}_{a_2+t_2}\cup K^{(a_3)}_{a_3+t_3}$ be a vertex disjoint union of cliques.  Any $2$-coloring of the
integers with no monochromatic translate of any set in $W'(S)$ is periodic with
period $2a_2$.
\end{lemma}
The proofs of Lemmas \ref{3cliques3:lem} and \ref{3cliques4:lem} are similar to that of Lemma \ref{3cliques2:lem}.  For Lemma
\ref{3cliques3:lem}, just as in the proof of Lemma \ref{3cliques2:lem}, there exist two integers with a
difference of $a_1 - a_2$ which must be the same color, say $C(0)=C(a_1-a_2)=R$.  Then $C(a_3 - a_2)=C(a_1 + a_3)=B$ (to avoid red translates of
$\{a_3,a_2,a_1\}$ and $\{-a_1,-a_2,a_3\}$ respectively).  Then $C(2a_3)=C(a_1 - a_2+ 2a_3)=R$, and so on, eventually showing that the coloring on $Z_i$ has period $2a_3$, for each $0\leq i\leq 2e-1$. For Lemma \ref{3cliques4:lem}, just as in the proof of Lemma \ref{3cliques2:lem}, there exist two integers
with a difference of $a_1-a_3$ which have the same color.  If $C(0)=C(a_1-a_3)=R$ then
$C(a_2-a_3)=C(a_1+a_2)=B$, so then $C(2a_2)=C(a_1+2a_2-a_3)=R$, and so on.

\begin{proof}[Proof of Theorem \ref{3cliques:thm}]
Suppose $a_1>a_2>a_3$ and $a_1 = p_12^{r_1}$, $a_2 = p_22^{r_2}$, $a_3 = p_32^{r_3}$ with $p_1,p_2,p_3$ odd and $r_1,r_2,r_3\geq 0$. Let $S=K^{(a_1)}_{a_1+t_1}\cup K^{(a_2)}_{a_2+t_2}\cup K^{(a_3)}_{a_3+t_3}$  be a vertex disjoint union of cliques, with $t_1$, $t_2$, $t_3\geq 0$.

Clearly if any pair of the cliques in $S$ is not $2$-Ramsey, then neither is
$S$.  So we just need to show that if each pair of cliques is $2$-Ramsey
then so is $S$.

Case 1.  $r_1 = r_2 = r_3$

By the above lemmas, any $2$-coloring of the integers
which does not have a monochromatic translate of any set in $W'(S)$ has period
$2a_1$, has period $2a_2$, and has period $2a_3$.  Hence it has period $d$ where
$d=\gcd(2a_1,2a_2,2a_3)$.  Thus there is a monochromatic translate of the
set $\{0,a_2-a_3,a_1-a_3\}$ from $W'(S)$, since $a_2-a_3$ and $a_1-a_3$ are multiples of $d$ (in fact there
are monochromatic translates of every set from $P'(S)$).

Case 2.  $r_1 = r_2\neq r_3$ and either
(i) $t_3 = 1$ and $r_1,r_2,r_3\geq 2$;
(ii) $t_1 = t_2 = 1$ and $r_1,r_2,r_3 \geq 2$;
(iii) $t_3 = 0$ or one of $t_1$ and $t_2$ is $0$ and the other is at most $1$ (with restrictions on the exponents according to Theorem \ref{2cliques:thm})

By Lemma \ref{3cliques3:lem}, if there is a $2$-coloring $C$ of the integers with no monochromatic translate of any set in $W'(S)$ then $C$ has period $2a_3$.  For subcase
(i), since $t_3=1$, $\{-a_2,j,j+2,a_1\}$ and $\{j,j+2,a_2,a_1\}$ are both translates of sets in $W'(S)$ for all $j\in [-a_3,a_3-2]_2$. As in the proof of Lemma \ref{3cliques2:lem} there exist two
integers with difference $a_1-a_2$ with the same color, say $R$.  Due to the $a_3$
forbidden sets containing $a_2$ listed above, no two consecutive even
integers in the period $2a_3$ coloring $C$ can be colored $R$. As in the  the proof
of Lemma \ref{3cliques2:lem}, there also exist two integers with difference $a_1 + a_2$
with the other color, $B$, so due to the $a_3$ forbidden sets containing
$-a_2$ listed above, no two consecutive even integers can be colored $B$.  That means $C$ must alternate colors on the
even integers. However, then $\{a_3,a_2,a_1\}$ is monochromatic because $a_1,a_2,a_3$ are
all multiples of $4$.

For subcase (ii), since $t_1 = t_2 = 1$, the set $\{0,2\}$ is in $W'(S)$: to see this take an automorphism given by flippling $(a_1-a_3)/2$ coordinates in the first clique and $(a_2-a_3)/2$ coordinates in the second clique. So the
only way to avoid a monochromatic translate is to alternate colors on the even
integers which, as in subcase (i), produces a monochromatic translate after all.

The proof of subcase (iii) is similar (but easier).

Case 3.  $r_1 = r_3\neq r_2$.  Subcase (i) is exactly as in Case 2.  For
subcase (ii), since $t_1 = t_3 = 1$, a translate of each of the sets
$\{j,j+2,a_2\}$ and $\{-a_2,j,j+2\}$ is in $W'(S)$ for each $j\in [-a_3,a_3 - 2]_2$,
and now an argument identical to the one in the proof of Theorem \ref{2cliques:thm} (3) for
the case $a_1 < a_2$ and $t_1 = 1$, produces a monochromatic translate after
all.

Case 4. $r_2 = r_3\neq  r_1$.  Almost identical to Case 3.

Case 5.  $r_1, r_2, r_3$ all distinct and greater than or equal to $2$, at
least two of $t_1,t_2,t_3$ equal to $0$ or $1$.   If $t_1$ and $t_2$ are equal to $0$
or $1$ then the set $\{0,2\}$ is in $W'(S)$ and we can finish as in Case 1(ii).
If $t_1 = t_3=1$  then translates of $\{j,j+2,a_2\}$ and $\{-a_2,j,j+2\}$ are in $W'(S)$
for each $j\in [-a_3,a_3-2]_2$.  We know that any coloring candidate has two
consecutive even integers with the same color, say $0$ and $2$ are colored $R$, so
$[a_2-a_3+2,a_2+a_3]_2$ is all $B$, so $[-2a_3+2,2a_3-2]_2$ is all R, and so on, producing
arbitrarily long sequences of consecutive even integers with the same
color, an impossibility.  The other possibilities in Case 5 are similar.
\end{proof}

\subsection{Arbitrary unions of cliques}
There is no analogue of Theorem \ref{3cliques:thm} for the disjoint union of four cliques of
different weights.  For example, if $S$ is the disjoint union of cliques of
weights $1$,$5$,$7$,$9$ then, no matter what the orders of the cliques may be, $S$
is $2$-Ramsey (this can be verified rather laboriously by hand by considering the 16 sets in $P'(S)$).  However, if $S$ is the disjoint union of cliques of
weights $1$, $5$, $7$, $11$,  and if the orders are large enough so that $W'(S)=P'(S)$, (so by Proposition \ref{7:prop}, orders at least 1, 9, 13, 21 respectively)
then $S$ is not $2$-Ramsey.  The period $38$ coloring of the integers obtained
by repeating the sequence $RRRRBBBRRBRBRBRRBBB$ on the even integers, and on
the odd integers, has no monochromatic translate of any of the $13$
sets in $P'(S)$ (in fact these colorings are the only colorings of the
integers with no monochromatic translate of any set in $W'(S)$).  By Theorem
\ref{3cliques:thm} the disjoint union of any three of these four cliques is $2$-Ramsey, no
matter what the orders of the cliques may be.

Which disjoint unions of $s$ cliques are not $2$-Ramsey, but the disjoint
union of any $s-1$ of the cliques is 2-Ramsey?  By Theorem \ref{3cliques:thm}, none with
$s=3$.  By our next result, none if s is sufficiently large.

If $S$ is the vertex disjoint union of $s$ cliques $ K_{a_1+t_1}^{(a_1)},\ldots, K_{a_s+t_s}^{(a_s)}$ where $a_i$ is odd and $t_i \leq 1$ for each $i$, then $S$ is $2$-Ramsey.  This is so
because $\{0,2\}\in W'(S)$ (take $b_i = (a_i + 1)/2$ for each $i$), but the
only $2$-coloring of the integers with no monochromatic translate of $\{0,2\}$, is one that
alternates colors on the even integers and so contains a  monochromatic translate of
the set $\{x_1a_1,\ldots,x_sa_s\}$ obtained by letting $x_i = 1$ if $a_i\equiv  1 \mod 4$, and $x_i=-1$ if $a_i \equiv 3 \mod 4$, since any pair of elements from this set differ by a multiple of $4$. 

Our final result (Theorem \ref{lll:thm}) tells us that if we require $t_i\geq 2$ for each
$i$ then, for sufficiently large $s$, the vertex disjoint union of $s$ cliques
of different weights cannot be $2$-Ramsey.  First we show that to prove this
we need only consider configurations $S$ where $a_i$ is odd for all $i$.

\begin{prop}\label{pari:prop}
Let $S=K_{a_1+t_1}^{(a_1)}\cup \cdots\cup K_{a_s+t_s}^{(a_s)}$, be a vertex disjoint union of $s$ cliques. For a positive integer $m$, let $S^{m}=K_{e_1+u_1}^{(e_1)}\cup\cdots\cup K_{e_s+u_s}^{(e_s)}$, be a vertex disjoint union of $s\geq 2$ cliques with $e_i= ma_i$ for each $i$.
\begin{itemize}
\item[(a)] If $u_i \geq 2$ for each $i$ and $S^{m}$ is $2$-Ramsey, then so is $S$ (for all values of the $t_i$'s).
\item[(b)] If $t_i \geq a_i - 1$ for each $i$ and $S$ is $2$-Ramsey, then so is $S^{m}$ (for all values of the $u_i$'s).
\end{itemize}
\end{prop}
\begin{proof}  We note that the sets in $P'(S^{m})$ are obtained by multiplying
each element in each set in $P'(S)$ by $m$.

For (a) suppose that $S^m$ is $2$-Ramsey. Since $u_i\geq 2$ for $1\leq i \leq s$, Theorem \ref{2cliques:thm} implies that $r_1=r_2=...=r_s$, where $2^{r_i}$ is the largest power of two that divides $a_i$. In particular all the $a_i$ are of the same parity. Now, for a contradiction, suppose that $S$ is not $2$-Ramsey and take a coloring $c$ of the integers avoiding all monochromatic translates of sets from $P'(S)$. Since the $a_i$ are all of the same parity the sets in $P'(S)$ only contain even integers. Hence the sets in $P'(S^m)$ only contain numbers congruent to $0$ mod $2m$ and so their translates lie in a congruence class mod $2m$. 

 Since each $u_i\geq 2$, Lemma \ref{embed:lem} implies that for any embedding $\psi:\Vd\to \Vd$,  $W(\psi(S^m))$ either contains a translate of $\{0,2,4\}$ or it contains a translate of  $A^m\in P'(S^m)$. Thus, if we construct a coloring $c'$ of the integers avoiding all monochromatic translates of of sets in $P'(S^m)$ and $\{0,2,4\}$ then $S^m$ is not $2$-Ramsey, a contradiction.

We can define such a coloring as follows: for integers $j,n$, with $0\leq j \leq 2m-1$, let $c'(2mn+j)=c(2n)$, if $j \equiv 0,1 \mod 4$ and $c'(2mn+j) \neq c(2n)$ if $j\equiv 2,3 \mod 4$. For any $m>1$, $c'$ avoids monochromatic translates of $\{0,2,4\}$ (if $c'(x)=c'(x-2)$ then $x\equiv 0,1$ mod $2m$ and so $c'(x+2)\neq c'(x)$). Moreover $c'$ restricted to any mod $2m$ congruence class gives a restriction of $c$ or its complement to the even integers. Since any monochromatic translate under $c'$ of a set $A^m\in P'(S^m)$ lies in a congruence class mod $2m$ it would correspond to a monochromatic translate under $c$ of a set $A\in P'(S)$, but $c$ contains no such monochromatic translates.

For (b) suppose $S^{m}$ is not $2$-Ramsey and let $c'$ be a coloring of the integers with no
monochromatic translate of any set in $P'(S^{m})$.  Define a coloring $c$ on the integers
by $c(n)=c'(mn)$.  Clearly $c$ does not produce a monochromatic translate of
any set in $P'(S)$, and since $P'(S)=W'(S)$ (by Proposition \ref{7:prop}), $S$ is not
$2$-Ramsey.
\end{proof}

\begin{theorem}\label{lll:thm}
If $S$ is the vertex disjoint union of $s\geq 39$ cliques $K_{a_1+t_1}^{(a_1)},\ldots,K_{a_s+t_s}^{(a_s)}$ contained in $\Q_d$, and each $t_i\geq 2$ then $S$ is not 2-Ramsey
\end{theorem}
\begin{proof}
We will use a probabilistic argument employing the Lov\'asz Local Lemma \cite{LLL} (see Lemma \ref{lll:lem} below).

Let $S$ be a vertex disjoint union of $s\geq 39$ cliques $K_{a_1+t_1}^{(a_1)},\ldots,K_{a_s+t_s}^{(a_s)}$, with each $t_i\geq 2$. Suppose, for a contradiction,  that $S$ is $2$-Ramsey. By Proposition \ref{pari:prop} (a) we may suppose that $\gcd(a_1,\ldots,a_s)=1$. If any pair of the $a_i$ are of different parities then $S$ is trivially not $2$-Ramsey (simply color all even layers red and all odd layers blue). So we may suppose that $a_1<a_2<\cdots <a_s$ are all odd and in particular $a_{i+1}-a_i\geq 2$ for $1\leq i \leq s-1$.  By Lemma \ref{translate:lem}, $W^*(S)$ is $2$-translate-Ramsey. Hence, by Lemma \ref{fin:lemma}, there exists $n_T$ such that any $2$-coloring of $[n_T]$  contains a monochromatic translate of $D\in W^*(S)=\{W(\psi(S)):\psi:\Vd\to \Vd\tr{ is an embedding}\}$.

 Since each $t_i\geq 2$, Lemma \ref{embed:lem} implies that for any embedding $\psi:\Vd\to \Vd$,  $W(\psi(S))$ either contains a translate of $\{0,2,4\}$ or it contains a translate of $\{x_1a_1,x_2a_2,\ldots,x_sa_s\}$, for some choice of signs $x_1,\ldots,x_s\in \{-1,+1\}$. To show that $S$ is not $2$-Ramsey it is sufficient to prove that there exists a coloring of $[n_T]$ with no monochromatic translate of $\{0,2,4\}$ or $\{x_1a_1,\ldots,x_sa_s\}$, for any choice of signs. We will do this by defining a random 2-coloring of the integers and showing that with positive probability no translate of sets of the above types are found in the restriction of this coloring to $[n_T]$.

Define a random coloring of the integers $c:\mathbb{Z}\to \{R,B\}$ as follows. For each $i\in \mathbb{Z}$ such that $i\equiv 0\tr{ or }1 \mod 4$, toss a fair coin (all coin tosses are independent). If the coin toss is heads set $c(i)=R$ and $c(i+2)=B$ otherwise set $c(i)=B$ and $c(i+2)=R$. We refer to each pair $(i,i+2)$ of integers colored in this way as a \emph{block}.

Note that if $y_1,y_2,\ldots,y_k$ are distinct integers no pair of which differ by exactly two then they are all colored independently. Moreover for any choice of colors $c_1,\ldots c_k\in \{R,B\}$ we have \[
\Pr[c(y_1)=c_1,c(y_2)=c_2,\ldots,c(y_k)=c_k]=2^{-k}.\]
The coloring has the property that for any $x\in \mathbb{Z}$ it is not true that $c(x)=c(x+2)=c(x+4)$ (since either $(x,x+2)$ or $(x+2,x+4)$ is a block). Hence no translate of $\{0,2,4\}$ is monochromatic.

For each integer $b$ 
let $R_b$ be the event that there exists a choice of signs $x_1,\ldots,x_s\in\{-1,+1\}$ such that $\{x_1a_1,\ldots,x_sa_s\}+b$ is red. Let $E^i_b$ be the event that at
least one of $b - a_i$ and $b + a_i$ is red.  Then
\[       \Pr[R_b] = \Pr[E_b^1\wedge E^2_b\wedge \cdots \wedge E_b^s].\]
Clearly $\Pr[E^i_b] = 3/4$ unless $i=1$ and $a_1 = 1$ (in which case it is
equal to $1$ if $(b-1,b+1)$ is a block, and $3/4$ otherwise).  We note that $(b +a_i, b + a_{i+1})$ is a block iff $(b - a_{i+1}, b - a_i)$ is a block, since
$a_{i+1} - a_i = 2$ implies $(b + a_{i+1}) - (b - a_i)$ is a multiple of $4$.
If $i<j$ and $b + a_i$, $b + a_j$  are in different blocks, then $E_b^i$ and $E_b^j$ are independent, while if they are in the same block then $j=i+1$ and $\Pr[E_b^i \wedge E_b^j]= 1/2$. 

Hence if $(b - a_1, b + a_1)$ is not a
block, and there are precisely $t$ blocks of the form $(b+a_i,b+a_{i+1})$, for some $1\leq i \leq s-1$, then
\[
       \Pr[R_b] = \frac{1}{2^t}\left(\frac{3}{4}\right)^{s-2t} \leq\left(\frac{3}{4}\right)^s,\]
while if $(b - a_1,b + a_1)$ is a block, then $\Pr[R_b] \leq (3/4)^{s-1}$.
Hence this last inequality holds no matter what.

For an integer $b$ let $M_b$ be the event that there exists a choice of signs $x_1,\ldots,x_s\in\{-1,+1\}$ such that $\{x_1a_1,\ldots,x_sa_s\}+b$ is monochromatic. By symmetry we have $\Pr[M_b]\leq 2(3/4)^{s-1}$.

Our next aim is to show that the event $M_b$ is independent of ``most'' other events $M_{b'}$, in the following sense.

\textbf{Claim:} $M_b$ is independent of all but at most $6s^2$ events $M_{b'}$.

For any integer $b$ let $D_b=\{\pm a_1,\pm a_2,\ldots, \pm a_s\}+ b$.
Let $b$ be fixed. We first  count the number of ways to choose $b'\neq b$ such that
$D_b\cap D_{b'}\neq \emptyset$.  If $D_b\cap D_{b'}\neq \emptyset$ and $b'\neq b$ then there exist $u,v\in \{\pm a_1,\ldots,\pm a_s\}$ such that $b'=b+u-v$. Now $b'\neq b$ implies that $u\neq v$. If we suppose also that $u\neq -v$ then there are $2s(2s-2)$ such ordered pairs $(u,v)$, but they produce at most $s(2s-2)$ distinct values of $b'$ (since $(u,v)$ and $(-v,-u)$ produce the same value of $b'$).  There are at most $2s$ other values of $b'$ produced when $u=-v$, so  there are a total of at most $2s^2$ distinct
values of $b'$ such that $D_b\cap D_{b'}\neq \emptyset$.  Since $M_b$ is
independent of all $M_{b'}$ except those for which there exist $x \in D_b$ and $y \in
D_{b'}$ such that $x \in \{y-2,y,y+2\}$, there are at most $6s^2$
values of $b'$ such that $M_b$ and $M_{b'}$ are dependent.


It is straightforward to check that for $s\geq 39$ we have $2(6s^2+1) e\left(\frac{3}{4}\right)^{s-1}<1$. Hence, by the Lov\'asz Local Lemma (Lemma \ref{lll:lem}), with non-zero probability $c$ gives a coloring of $[n_T]$ with no monochromatic translate of $\{x_1a_1,\ldots,x_sa_s\}$ for any choice of signs $x_1,\ldots,x_s\in \{-1,+1\}$. Hence $S$ is not $2$-Ramsey.
\end{proof}
\begin{lemma}[Erd\H os--Lov\'asz \cite{LLL}]\label{lll:lem}
Let  $A_1,\ldots,A_k$ be events in a probability space that each occur with probability at most $p$. If each event is independent of all but at most $d$ other events and $ep(d+1)\leq 1$ then there is a non-zero probability that none of the events occur.
\end{lemma}
\section{Questions}
Given Theorem \ref{lll:thm}, a natural question to ask is: do there exist 2-Ramsey subsets of $\Vd$ that cannot be embedded into a small number of layers? To make this precise we define $l(S)$ to be the smallest number layers into which $S\subseteq \Vd$ can be embedded:
\[l(S)=\min_{B\in W'(S)}|B|.\]
\begin{question}
Do there exist subsets $S_d\subseteq \Vd$ such that $S_d$ is 2-Ramsey and $\lim_{d\to \infty}l(S_d)=\infty$?
\end{question}

Another natural question to ask is: how large can a $2$-Ramsey subset of $\Vd$ be? By Ramsey's theorem examples of size $\binom{d}{\lfloor d/2\rfloor}$ exist.
\begin{question}
If $S\subseteq \Vd$ is $2$-Ramsey how large can $|S|$ be?
\end{question}
\section*{Acknowledgements}
We wish to thank both referees for their careful reading of the original version of this paper. Their detailed comments and suggestions were extremely helpful.

\end{document}